\input ./style/arxiv-general.cfg
\documentclass[MSNbibl,number,citesort,dvips]{arxbj}
\makeatletter
   \@ifpackageloaded{graphicx}{}{\usepackage{graphicx}}
\makeatother
\usepackage{mathbh,upgreek}
\usepackage{accents}

\volume{22}
\issue{3}
\pubyear{2016}
\firstpage{1709}
\lastpage{1728}
\doi{10.3150/15-BEJ708}
\docsubty{FLA}

\makeatletter
\def\sfrac#1#2{#1/#2}

\def\afrac#1#2{#1/(#2)}
\def\vafrac#1#2{(#1)/(#2)}
\def\sklfrac#1#2{(#1/#2)}

\newcommand{\ee}{\mathrm{e}}
\renewcommand{\mathring}[1]{\accentset{\circ}{#1}}
\newcommand{\rrvert}{\vert}
\newcommand{\rrVert}{\Vert}
\newcommand{\llvert}{\vert}
\newcommand{\llVert}{\Vert}
\renewcommand{\mid}{|}
\newcommand{\mathds}{\mathbh}
\newcommand{\dd}{\mathrm{d}}
\newcommand{\dif}{\mathrm{d}}
\newtheorem{Th}{Theorem}[section]
\newproclaim{Def}[Th]{Definition}
\newtheorem{Lem}[Th]{Lemma}
\newtheorem{Prop}[Th]{Proposition}
\newtheorem{Cor}[Th]{Corollary}
\newremark{Rem}[Th]{Remark}
\newremark{Ex}[Th]{Example}
\newcommand{\Var}{\operatorname{Var}}
\def\E{{\mathbb{E}}}
\def\N{{\mathbb{N}}}
\def\R{{\mathbb{R}}}
\makeatother

\begin{document}
\begin{frontmatter}

\title{Approximation of improper priors}
\runtitle{Approximation of improper priors}

\begin{aug}
\author[A]{\inits{C.}\fnms{Christele}~\snm{Bioche}\corref{}\thanksref{e1}\ead[label=e1,mark]{Christele.Bioche@math.univ-bpclermont.fr}}
\and
\author[A]{\inits{P.}\fnms{Pierre}~\snm{Druilhet}\thanksref{e2}\ead[label=e2,mark]{Pierre.Druilhet@univ-bpclermont.fr}}
\address[A]{Laboratoire de Math\'{e}matiques, UMR CNRS~6620, Clermont
Universit\'e, Universit\'{e} Blaise Pascal, 63177~Aubi\`ere Cedex,
France.
\printead{e1}, \printead*{e2}}
\end{aug}

%
\received{\smonth{7} \syear{2014}}
%
\revised{\smonth{1} \syear{2015}}

%
\begin{abstract}
We propose a convergence mode for positive Radon measures which allows
a sequence of probability measures to have an improper limiting measure.
We define a sequence of vague priors as a sequence of probability
measures that converges to an improper prior.
We consider some cases where vague priors have necessarily large
variances and other cases where they have not.
We study the consequences of the convergence of prior distributions on
the posterior analysis.
Then we give some constructions of vague priors that approximate the
Haar measures or the Jeffreys priors.
We also revisit the Jeffreys--Lindley paradox.
\end{abstract}

%
\end{frontmatter}

\section{Introduction}

Improper priors such as flat priors (Laplace \cite{MR1400403}),
Jeffreys priors (Jeffreys \cite{MR0017504}), reference priors
(Berger \textit{et~al.} \cite{MR2502655}) or the Haar measures (Eaton \cite{MR1089423}) are often used in Bayesian
analysis when no prior information is available.
The posterior distribution is obtained by applying the formal Bayes rule.
There are several approaches to justify the use of improper priors in
statistics.
Taraldsen and Lindqvist \cite{MR2757006} explain how the theory of conditional
probability spaces developed by R{\'e}nyi \cite{MR0264723} is related to a theory for
statistics that includes improper priors.
Their article is based on a generalization of Kolmogorov's theory to
the $\sigma$-finite measures.
They show in particular by examples that this theory is different from
the alternative theory of improper priors provided by Hartigan \cite{MR715782}.
For many authors, the inference based on an improper prior $\Pi$ is
legitimated as limit of inferences based on proper priors $\Pi_n$.
However, there are several ways to define this limit.
For example, Jeffreys \cite{MR0187257}, Stone \cite{MR0266359}, Bernardo and Smith (\cite{MR1274699},
Proposition 5.11), Jaynes \cite{MR1992316} consider the convergence, for any
given observation $x$,
of the posterior distributions $\Pi_n(\cdot\mid x)$ to $\Pi(\cdot\mid x)$
for some convergence mode such as total variation.
Stone \cite{MR0149586} consider a convergence mode involving both the posterior
distribution and the marginal distribution.

All these convergence modes are related to the statistical model
through the likelihood.
Moreover, there is no standard convergence mode such that a sequence
$\Pi_n$ of proper priors may converge to an improper prior $\Pi$
independently on the statistical model.
Consider, for example, a sequence of normal distributions $\mathcal
N(0,n)$ with zero mean and variance equal to $n$; it is often admitted
that this sequence converges to the Laplace prior since for
many statistical models the Bayes estimate related to $\mathcal N(0,n)$
converges to the Bayes estimate for the Laplace prior.
A~question then arises:
does the limiting behaviour of a sequence of proper priors depend on
the statistical model?
Is there any intrinsic convergence mode?

The aim of this paper is to define a convergence mode on the set of
prior distributions without reference to any statistical model.
In Section~\ref{section2}, we define this convergence mode.
We show that a sequence of vague priors is related to at most one
improper prior.
We also show that any improper distribution can be approximated by
proper distributions and reciprocally.
In Section~\ref{section3}, we give some conditions on the likelihood
to derive convergence of posterior distributions and Bayesian
estimators from the convergence of prior distributions.
In Section~\ref{nvellesection3}, we give some examples of
construction of sequences of probability measures which converge to
improper priors such as the Haar measure or the Jeffreys prior.
In Section~\ref{section4}, we give a special interest in the
convergence of Beta distributions.
In Section~\ref{section5}, we revisit the Jeffreys--Lindley paradox in
the light of our convergence mode.

\section{Definition, properties and examples of $q$-vague convergence}\label{section2}

Let $X$ be a random variable and assume that $X\mid\theta\sim P_{\theta
}, \theta\in\Theta$. We assume that $\Theta$ is in $\mathbb{R}$,
$\mathbb{R}^p$ with $p>1$, or a countable set.
In the Bayesian paradigm, a prior distribution $\Pi$ is given on
$\Theta$.
In this article, we always assume that a prior $\Pi$ is a positive
Radon measure, {that is}, a positive measure which is finite on
compact sets.
So, a prior may be proper or improper.
We denote by $\pi$ the density function with respect to the Lebesgue
measure in the continuous case and the counting measure in the discrete
case, or more generally to some $\sigma$-finite measure.
If $\Pi$ is a probability measure, we can use the Bayes formula to
write the posterior density:
%
%
\begin{equation}
\label{eqformuleBayes} \pi(\theta\mid x)=\frac{f(x\mid\theta) \pi
(\theta)}{\int_\Theta
f(x\mid\theta)\pi(\theta)\,\dd\theta},
\end{equation}
where $f(x\mid\theta)$ is the likelihood function.

If $\Pi$ is an improper measure but $\int_\Theta f(x\mid\theta)\pi
(\theta)\,\dd\theta< +\infty$, we can formally apply the Bayes
formula to get a posterior distribution which will be proper.
Now, if we replace $\Pi$ by $\alpha\Pi$, for $\alpha>0$, we obtain
the same posterior distribution.
So, in this case, the posterior distribution is proper and independent
of changes in the scaling of the prior.
If $\Pi$ is an improper measure with $\int_\Theta f(x\mid\theta)\pi
(\theta)\,\dd\theta= +\infty$, we cannot apply the Bayes formula.
But in this article, we allow posterior distribution to be improper,
and in this case we will define it by $\pi(\theta\mid x)=f(x\mid\theta
) \pi(\theta)$ up to within a scalar factor.

We denote by $\mathcal{C}_K(\Theta)$ the space of real-valued
continuous functions on $\Theta$ with compact support and by $\mathcal
{C}_K^+(\Theta)$ the positive functions in $\mathcal{C}_K(\Theta)$.
When there is no ambiguity on the space, they will be simply denoted by
$\mathcal{C}_K$ or $\mathcal{C}_K^+$.
We also introduce the notation $\mathcal{C}_b(\Theta)$ for the space
of bounded continuous functions on $\Theta$,
and $\mathcal{C}_0(\Theta)$ for the space of continuous functions $g$
such that for all $\varepsilon>0$, there exists a compact $K\subset
\Theta$ such that for all $\theta\in K^c$, $g(\theta)<\varepsilon$.
We use the notation $\Pi(h)=\int_{\Theta}h\,\dd\Pi$ where $h$ is a
measurable real-valued function, and $\llvert \Pi\rrvert =\Pi(1)=\int
_{\Theta}\dd
\Pi$, the total mass of $\Pi$.

We recall the two classic kinds of convergence of measures (Bauer \cite{MR1897176}).
A sequence of probability measures $\{\Pi_n\}_n$ converges narrowly
(also said weakly) to a probability measure $\Pi$ if,
for every function $\phi$ in $\mathcal{C}_b(\Theta)$, $\{\Pi_n(\phi
)\}_n$ converges to $\Pi(\phi)$.
A sequence of positive Radon measures $\{\Pi_n\}_n$ converges vaguely
to a positive Radon measure $\Pi$ if, for every function $\phi$ in
$\mathcal{C}_K(\Theta)$, $\{\Pi_n(\phi)\}_n$ converges to $\Pi
(\phi)$.
We also recall a characterization of vague convergence for a sequence
of probability measures which will be useful later in the article.

%
\begin{Lem}[(Billingsley \cite{MR830424}, page~393)]\label{lienconvvagueconvetroite}
If $\{\Pi_n\}_n$ is a sequence of probability measures and $\Pi$ is a
probability measure,
then $\{\Pi_n\}_n$ converges vaguely to $\Pi$ iff for all $g\in
\mathcal{C}_0(\Theta)$, $\{\Pi_n(g)\}_n$ converges to $\Pi(g)$.
\end{Lem}

\subsection{Convergence of prior distribution sequences}\label{sectiondefconv}

In this section, we define a new convergence mode for sequences of
positive Radon measures.
The aim is to propose a formalization of an usual practice which
consists of approximate an improper prior with a sequence of proper priors.

%
\begin{Def}\label{caractconv}
A sequence of positive Radon measures $\{\Pi_n\}_n$ is said to
converge $q$-vaguely to a positive Radon measure $\Pi$ if there exists
a sequence of positive real numbers $\{a_n\}_n$ such that
$\{ a_n\Pi_n\}_n$ converges vaguely to $\Pi$.
\end{Def}

Let us justify this definition.
In equation~(\ref{eqformuleBayes}), if we replace $\Pi$ by $\alpha
\Pi$, for $\alpha>0$, we obtain the same posterior distribution,
which means that the prior distribution is defined up to within a
scalar factor.
So, it is natural to define the equivalence relation $\sim$ on the
space of positive Radon measures by
%
%
\begin{equation}
\label{releq} \Pi\sim\Pi'\quad\Longleftrightarrow\quad\exists\alpha>0\qquad
\mbox{such that }\Pi=\alpha\Pi'.
\end{equation}
Then it is natural to define the quotient space of positive Radon measures
by the equivalence relation $\sim$.
We denote by $\overline{\Pi}$ the equivalence class of $\Pi$,
{that is}, $\overline{\Pi} =\{\widetilde{\Pi}/\exists
\alpha>0,\widetilde{\Pi}=\alpha\Pi\}$.
The $q$-vague convergence corresponds to the standard quotient topology
on this quotient space.

%
\begin{Rem}
One referee pointed out that similar quotient spaces for $\sigma
$-finite measures were considered by Taraldsen and Lindqvist \cite{Tara2010} to
define conditional measures.
\end{Rem}

%
\begin{Prop}\label{remcaractconv}
Let $\{\Pi_n\}_n$ and $\Pi$ be positive Radon measures.
The sequence $\{\Pi_n\}_n$ converges $q$-vaguely to $\Pi$ iff $\{
\overline{\Pi}_n\}_n$ converges to $\overline{\Pi}$ for the
quotient topology.
\end{Prop}

\begin{pf}\label{appendixproofcaract}
$\bullet$
Direct part: Assume that $\lim_{n\to\infty} \overline{\Pi}_n =
\overline{\Pi}$. The space of positive
Radon measures is a metrisable space so it admits a countable
neighbourhood base.
Thus, there exists a decreasing sequence of open sets $\{O_i\}_{i\in
\mathbb{N}}$ in the space of positive Radon measures such that for all
$i\in\N, \Pi\in O_i$ and $\bigcap_{i\in\mathbb{N}}O_i=\{\Pi\}$.
So, for all $i\in\mathbb{N}$, $\overline{\Pi}\in\overline{O}_i$.
For any $O_i$, there exists $N_i$ such that for all $n>N_i$, $\overline
{\Pi}_n\in\overline{O}_i$.
Without lost of generality, we can choose $N_i$ such that $N_i>N_{i-1}$.
For all $n$ such that $N_i\leq n<N_{i+1}$, $\Pi_n\in\mathcal
{C}(O_i)$ where $\mathcal{C}(O_i)=\{\lambda x$ with $\lambda>0$ and
$x\in O_i\}$,
{that is}, for all $n$ such that $N_i\leq n < N_{i+1}$,
there exists $a_n>0$ such that $a_n\Pi_n\in O_i$.
Moreover, since $\bigcap_{i\in\mathbb{N}} O_i=\{\Pi\}$, $\lim_{n\to
\infty}a_n\Pi_n=\Pi$.

$\bullet$ 
Converse part: Assume that $\{a_n\Pi_n\}_n$ converges to $\Pi$.
Since the canonical mapping $\phi$ defined by
%
%
\begin{eqnarray}
\label{defphi} \phi\dvtx \mathcal{R} & \to& \mathcal{R}/\sim,
\nonumber
\\[-8pt]
\\[-8pt]
\nonumber
\Pi& \mapsto& \overline{\Pi},
\end{eqnarray}
where $\mathcal{R}$ is the space of positive Radon measures, is
continuous, $\{\phi(a_n\Pi_n)\}=\{\overline{\Pi}_n\}$ converges to
$\phi(\Pi)=\overline{\Pi}$.
%
\end{pf}

The following proposition shows that a sequence of prior measures
cannot converge $q$-vaguely to more than one limit up to within a
scalar factor.

%
\begin{Th}\label{unicitelimite}
Let $\{\Pi_n\}_n$ be a sequence of priors such that $\{\Pi_n\}_n$
converges $q$-vaguely to both $\Pi_a$ and $\Pi_b$,
then necessarily there exists $\alpha>0$ such that $\Pi_a=\alpha\Pi_b$.
\end{Th}

\begin{pf}
This is a direct consequence of Proposition~\ref{Hausdorffspace} that
states that $\overline{\mathcal{R}}$ is a Hausdorff space.
However, we give here a direct proof that does not involve abstract
topological concept.

Assume that $\{\Pi_n\}_n$ converges $q$-vaguely to both $\Pi_a$ and
$\Pi_b$. From Definition~\ref{caractconv},
there exist two sequences of positive scalars $\{a_n\}_n$ and $\{b_n\}
_n$ such that
$\{a_n\Pi_n\}_n$, respectively, $\{b_n\Pi_n\}_n$, converges vaguely to
$\Pi
_a$, respectively, $\Pi_b$.
We have to prove that $\Pi_b= \alpha\Pi_a$ for some positive scalar~$\alpha$.
Since $\Pi_a\neq0$ and $\Pi_b\neq0$, there exist $h_a$ and $h_b$ in
$\mathcal{C}_K^+$ such that $\Pi_a(h_a)>0$ and $\Pi_b(h_b)>0$. Put
$h_0=h_a+h_b$; we have $\Pi_a(h_0)>0$ and $\Pi_b(h_0)>0$.
Moreover, $\lim_{n\to\infty} a_n\Pi_n(h_0) = \Pi
_a(h_0)$ and $\lim_{n\to\infty} b_n\Pi_n(h_0) = \Pi
_b(h_0)$. So,
there exists $N$ such that for $n\geq N$, $a_n\Pi_n(h_0)>0$ and
$b_n\Pi_n(h_0)>0$.
For any $h$ in $\mathcal{C}_K$ and $n>N$,
$\lim_{n\to\infty}\frac{\Pi_n(h)}{\Pi_n(h_0)} = \lim_{n\to\infty}\frac
{a_n\Pi_n(h)}{a_n\Pi_n(h_0)} = \frac
{\Pi_a(h)}{\Pi_a(h_0)}$
and $\lim_{n\to\infty}\frac{\Pi_n(h)}{\Pi_n(h_0)} = \lim_{n\to\infty
}\frac{b_n\Pi_n(h)}{b_n\Pi_n(h_0)} = \frac{\Pi_b(h)}{\Pi_b(h_0)}$.
By uniqueness of the limit in $\mathbb{R}$, $\frac{\Pi_a(h)}{\Pi
_a(h_0)} = \frac{\Pi_b(h)}{\Pi_b(h_0)}$.
Therefore, $\Pi_a = \frac{\Pi_a(h_0)}{\Pi_b(h_0)}\Pi_b$. The
result follows.
\end{pf}

Theorem~\ref{espquotient} motivates to include the improper priors in
the theory since it shows these are obtained naturally from limits of
proper priors.
This can be compared with a completion of a metric space.

%
\begin{Th}\label{espquotient}
Any improper measure may be approximated by a sequence of probability
measures and
conversely, any proper measure may be approximated by a sequence of
improper measures.
\end{Th}

\begin{pf}
$\bullet$ Consider an improper measure $\Pi$ and $\{K_n\}_n$ an increasing
sequence of compacts such that $\Theta=\bigcup_n K_n$.
Then $\Pi_n=\Pi\mathds{1}_{K_n}$ is a proper measure so, $\frac
{1}{|\Pi_n|}\Pi_n$ is a probability measure.
Moreover, $\{\Pi_n\}_n$ converges vaguely to $\Pi$, so $\{\frac
{1}{|\Pi_n|}\Pi_n\}$ converges $q$-vaguely to $\Pi$.

$\bullet$ 
Let $\Pi$ be a probability measure.
Consider the sequence $\Pi_n=\Pi+\alpha_n\Pi'$ where $\Pi'$ is an
improper measure and $\{\alpha_n\}_n$ is a decreasing sequence which
converges to $0$.
Then, for all $n\in\mathbb{N}$, $\Pi_n$ is an improper measure and
$\{\Pi_n\}_n$ converges $q$-vaguely to $\Pi$.
%
\end{pf}

In many statistical models, there are several parameterizations of
interest. We show that the \mbox{$q$-}vague convergence is invariant by change
of parameterization.
Consider\vspace*{1pt} a new parameterization $\eta=h(\theta)$ where $h$ is a homeomorphism.
We denote by $\widetilde{\Pi}_n=\Pi_n\circ h^{-1}$ and $\widetilde
{\Pi}=\Pi\circ h^{-1}$ the prior distribution on $\eta$ derived from
the prior distribution on $\theta$.
The following proposition establishes a link between $q$-vague
convergence of $\{\Pi_n\}_n$ and $\{\widetilde{\Pi}_n\}_n$.

%
\begin{Prop}\label{chgmtparam}
Let\vspace*{1pt} $\{\Pi_n\}_n$ be a sequence of priors which converges $q$-vaguely
to $\Pi$.
Let $h$ be a homeomorphism and consider the parameterization $\eta
=h(\theta)$.
Then $\{\widetilde{\Pi}_n\}_n$ converges $q$-vaguely to $\widetilde
{\Pi}$.
\end{Prop}

\begin{pf}
From the change of variables formula,
$\int g(h(\theta))\,\dif\Pi_n(\theta)=\int g(\eta) \,\dif\widetilde
{\Pi}_n(\eta)$ and  $\int g(h(\theta))\,\dif\Pi(\theta)=\int g(\eta
) \,\dif\widetilde{\Pi}(\eta)$.
Moreover, if $\{\Pi_n\}_n$ converges $q$-vaguely to $\Pi$, from
Definition~\ref{caractconv} there exists $\{a_n\}_n$ such that $\{
a_n\Pi_n\}_n$ converges vaguely to $\Pi$.
Note that for all $g\in\mathcal{C}_K$, $g\circ h\in\mathcal{C}_K$.
So,\vspace*{1pt} for all $g\in\mathcal{C}_K$, $\lim_{n\to\infty} a_n \int g(h(\theta
)) \,\dif\Pi_n(\theta) = \int g(h(\theta
)) \,\dif\Pi(\theta)$,
{that is}, $\lim_{n\to\infty} a_n \int g(\eta
) \,\dif\widetilde{\Pi}_n(\eta) = \int g(\eta) \,\dif
\widetilde{\Pi}(\eta)$.
Thus, $\{\widetilde{\Pi}_n\}_n$ converges $q$-vaguely to $\widetilde
{\Pi}$.
\end{pf}

\subsection{Convergence when approximants are probabilities}

In this section, the sequence of approximants $\{\Pi_n\}_n$ is assumed
to be a sequence of probability measures.
Then we can establish some links between $q$-vague and narrow convergence.

Indeed, if $\{\Pi_n\}_n$ is a sequence of probabilities and $\Theta$
is a compact set, $q$-vague convergence is equivalent to narrow convergence.

More generally, we give a necessary and sufficient condition for the
narrow convergence of a sequence of probabilities which converges
$q$-vaguely to a probability.
We recall that a sequence of bounded measures $\{\Pi_n\}_n$ is said to
be tight if, for each $\varepsilon>0$, there exists a compact set $K$
such that, for all $n$, $\Pi_n(K^c)<\varepsilon$.

%
\begin{Prop}\label{lienconvq-vetnarrow}
Let $\{\Pi_n\}_n$ and $\Pi$ be probability measures such that $\{\Pi
_n\}_n$ converges $q$-vaguely to $\Pi$.
Then $\{\Pi_n\}_n$ converges narrowly to $\Pi$ iff $\{\Pi_n\}_n$ is tight.
\end{Prop}

\begin{pf}
Direct part: $\{\Pi_n\}_n$ converges narrowly to $\Pi$ a probability
measure so $\{\Pi_n\}_n$ is tight.

Converse part: Let us show that if $\{\Pi_{n_k}\}_k$ is any
subsequence of $\{\Pi_n\}_n$ which converges narrowly then $\{\Pi
_{n_k}\}_k$ converges to $\Pi$.
From Billingsley (\cite{MR830424}, Theorem 25.10), there exists a subsequence $\{
\Pi_{n_k}\}_{k}$ of $\{\Pi_n\}_n$ which converges narrowly to some
probability measure, say $\widetilde{\Pi}$.
Since $\{\Pi_{n_k}\}_{k}$ is a sequence of probabilities which
converges narrowly to $\widetilde{\Pi}$, from Definition~\ref
{caractconv}, $\{\Pi_{n_k}\}_{k}$ converges $q$-vaguely to $\widetilde
{\Pi}$.
So, from Theorem~\ref{unicitelimite}, there exists $\alpha>0$ such
that $\Pi= \alpha\widetilde{\Pi}$, but $\Pi$ and $\widetilde
{\Pi}$ are probabilities. So $\Pi=\widetilde{\Pi}$.
The result follows from Billingsley (\cite{MR830424}, Corollary of Theorem~25.10, page~346).
\end{pf}

Now, we also assume that the limiting measure $\Pi$ is an improper measure.
Then we can give a result about the sequence $\{a_n\}_n$ which will be
useful thereafter.

%
\begin{Lem}\label{aninf}
Let $\{\Pi_n\}_n$ be a sequence of probability measures and $\{a_n\}
_n$ a sequence of positive scalars such that
$\{a_n\Pi_n\}_n$ converges vaguely to $\Pi$. If $\Pi$ is improper,
then necessarily $\lim_{n\to\infty} a_n = +\infty$.
\end{Lem}

\begin{pf}
We assume that $\{a_n\Pi_n\}_{n}$ converges vaguely to $\Pi$
so, we have $\Pi(\Theta) \leq \liminf_n a_n \Pi
_n(\Theta)$ (see Bauer \cite{MR1897176}, Theorem 30.3).
But for all $n \in\N$, $\Pi_n(\Theta) = 1$ so $\Pi(\Theta
) \leq \liminf a_n$.
Moreover, $\Pi(\Theta) = +\infty$ so $\liminf_n a_n = +\infty$. The
result follows.
\end{pf}

%
\begin{Lem}[(Lang \cite{Lang}, page~38)]\label{existancefctinfind}
Let $E$ be $\R$, $\R^p$ with $p>1$ or a countable set,
for all compact $K_0 \subset (\bigcup_{n>0}\mathring{K}_n ) = E$,
there exists a function $h \in\mathcal{C}_K(E)$ such that $\mathds
{1}_{K_0} \leq h \leq 1$.
\end{Lem}

When a sequence of proper priors is used to approximate an improper
prior, the mass tends to concentrate outside any compact set.

%
\begin{Prop}\label{lemcompact}
Let $\{\Pi_n\}_n$ be a sequence of probability measures which
converges $q$-vaguely to an improper prior $\Pi$.
Then, for any compact $K$ in $\Theta$, $\lim_{n\to\infty
}\Pi_n(K)=0$, and consequently, $\lim_{n\to\infty}\Pi_n(K^c)=1$.
\end{Prop}

\begin{pf}
From Definition~\ref{caractconv}, there exists $\{a_n\}_n$ such that
$\lim_{n\to\infty} a_n\Pi_n(h) = \Pi(h)$ for any $h$
in~$\mathcal{C}_K$.
From Lemma~\ref{aninf}, $\lim_{n\to\infty} a_n=+\infty$
whereas $\Pi(h)<+\infty$, so $\lim_{n\to\infty}\Pi_n(h)=0$.
Let $K_0$ be a compact set in $\Theta$. From Lemma~\ref
{existancefctinfind}, there exists a function $h\in\mathcal{C}_K$ such that
$\mathds{1}_{K_0}\leq h$.
So $\Pi_n(K_0)\leq\Pi_n(h)$ and $\lim_{n\to\infty
}\Pi_n(K_0)=0$.
Since $\Pi_n(K_0)+\Pi_n(K_0^c)=1$ for all $n\in\N$,
thus $\lim_{n\to\infty}\Pi_n(K_0^c)=1$.
\end{pf}

Many authors consider that few knowledge on the parameter is
represented by priors with large variance.
Here, we establish some links between the $q$-vague convergence of
priors and the convergence of the sequence of corresponding variances.

%
\begin{Prop}\label{divvar}
Let $\{\Pi_n\}_n$ be a sequence of probabilities on $\mathbb{R}$ such
that $\mathbb{E}_{\Pi_n}(\theta)$ is a constant.
If $\{\Pi_n\}_n$ converges $q$-vaguely to an improper prior $\Pi$
whose support is $\R$, then $\lim_{n\to\infty}\Var_{\Pi
_n}(\theta)= +\infty$.
\end{Prop}

\begin{pf}
Since $\E_{\Pi_n}(\theta)$ is constant, $\lim_{n\to\infty
}\Var_{\Pi_n}(\theta)=+\infty$
iff $\lim_{n\to\infty}\E_{\Pi_n}(\theta^2) = +\infty$.
For any $r>0$, we have $\mathbb{E}_{\Pi_n}(\theta^2)\geq\int_{[-r,r]^c}
\theta^2\,\dif\Pi_n(\theta)$
so $\mathbb{E}_{\Pi_n}(\theta^2) \geq r^2 \Pi_n([-r,r]^c)$.
From Proposition~\ref{lemcompact}, $\lim_{n\to\infty}\Pi
_n([-r,r]^c)=1$ and then $\lim_{n\to\infty}\mathbb{E}_{\Pi
_n}(\theta^2)\geq r^2$.
Since this holds for any $r>0$, $\lim_{n\to\infty}\mathbb
{E}_{\Pi_n}(\theta^2)=+\infty$.
\end{pf}

%
\begin{Cor}\label{cordivvar}
Let $\{\Pi_n\}_n$ be a sequence of probabilities with constant mean
which approximate the Lebesgue measure $\lambda_{\R}$.
Then, necessarily, $\lim_{n\to\infty}\Var_{\Pi_n}(\theta
) = +\infty$.
\end{Cor}

However, we will see in the examples in Section~\ref{sectionexemplesgammavar} that when we do not assume the expectation
to be constant; the variance does not necessarily diverge.

\subsection{Characterization of $q$-vague convergence}\label{cascontinu}

In this section, we establish several sufficient conditions for the
$q$-vague convergence of $\{\Pi_n\}_n$ to $\Pi$ through their
probability density function (p.d.f.).
When $\Theta$ is continuous, then $\pi_n$ and $\pi$ are the standard
p.d.f. with respect to the Lebesgue measure.
When $\Theta$ is discrete, then $\pi(\theta_0)=\Pi(\theta=\theta
_0)$, {that is}, $\pi$ is the p.d.f. with respect to the counting measure.

When $\Theta=\{\theta_i\}_{i\in I}$ is a discrete set with $I\subset
\mathbb N$, we give an easy-to-check characterization of the $q$-vague
convergence.

%
\begin{Prop}\label{caractconvmesdiscrete}
Let $\{\Pi_n\}_n$ and $\Pi$ be priors on $\Theta=\{\theta_i\}_{i\in
I}$, $I\subset\mathbb N$.
The sequence $\{\Pi_n\}_n$ converges $q$-vaguely to $\Pi$ iff
there exists a sequence of positive real numbers $\{a_n\}_n$ such that
for all $i\in I$, $\lim_{n\to\infty} a_n\pi_n(\theta
_i)=\pi(\theta_i)$.
\end{Prop}

\begin{pf}
It is a direct consequence of Definition~\ref{caractconv} applied to
the discreet case.
\end{pf}

Now, we consider the continuous case.

%
\begin{Prop} \label{corconvdensite}
Let $\{\Pi_n\}_n$ and $\Pi$ be continuous priors on $\Theta$ in $\R
$ or $\R^p$ with $p>1$.
Assume that:
\begin{longlist}[(2)]
\item[(1)] there exists a sequence of positive real numbers $\{a_n\}_{n}$ such
that the sequence $\{a_n\pi_n\}_{n}$ converges pointwise to $\pi$,

\item[(2)] there exists a continuous function $g\dvtx\Theta\rightarrow\mathbb
{R}^+$ and $N\in\mathbb N$ such that for all $n>N$ and $\theta\in
\Theta$, $a_n\pi_n(\theta) < g(\theta)$.
\end{longlist}
Then $\{\Pi_n\}_n$ converges $q$-vaguely to $\Pi$.
\end{Prop}

\begin{pf}
Let $h$ be in $\mathcal{C}_K(\Theta)$.
Then, $a_n h(\theta)\pi_n(\theta) \leq \llVert { h }
\rrVert g \mathds{1}_{K}(\theta)$
where $\llVert { h } \rrVert =\max_{\theta\in\Theta}
h(\theta)$.
Since $\llVert { h } \rrVert g \mathds{1}_{K}(\theta)$ is
Lebesgue integrable, by dominated convergence theorem,\break
$\lim_{n\to\infty} \int a_n \pi_n(\theta) \*h(\theta) \,\dif\theta= \int\pi
(\theta)h(\theta)\,\dif
\theta$.
\end{pf}

The following result will be useful to establish a result in
Section~\ref{section253}.

%
\begin{Prop} \label{thmconvdensite}
Let $\{\Pi_n\}_n$ and $\Pi$ be priors.
Assume that:
\begin{longlist}[(2$'$)]
\item[(1)] there exists a sequence of positive real numbers $\{a_n\}_{n}$ such
that the sequence $\{a_n\pi_n\}_{n}$ converges pointwise to $\pi$,

\item[(2$'$)] for any compact set $K$, there exists a scalar $M$ and some $N\in
\mathbb N$ such that for $n>N$, $\sup_{\theta\in K}a_n \pi_n(\theta
)<M$.
\end{longlist}
Then $\{\Pi_n\}_n$ converges $q$-vaguely to $\Pi$.
\end{Prop}

\begin{pf}
The proof is similar to the proof of Proposition~\ref{corconvdensite} with
$a_n \pi_n(\theta) h(\theta) \leq M \sup_{\theta\in
K} |h(\theta)| \mathds{1}_K(\theta)$.
\end{pf}

%
\begin{Rem}
Proposition~\ref{corconvdensite} and Proposition~\ref{thmconvdensite}
hold if $\pi(\theta)$ is the p.d.f. with respect to any
positive Radon measure.
\end{Rem}

\section{Convergence of posterior distributions and estimators}\label{section3}

Consider the model $X|\theta\sim P_{\theta}$, $\theta\in\Theta$.
We denote by $f(x\mid\theta)$ the likelihood.
The priors $\Pi_n$ on $\Theta$ represent our prior knowledge.
We always assume that $\int_{\Theta} f(x\mid\theta)\,\dif\Pi(\theta)>0$.

For a measure $\Pi$ and a measurable function $g$, we define the
measure $g\Pi$ by $g\Pi(f) = \Pi(gf) = \int f(\theta)
g(\theta) \,\mathrm{d}\Pi(\theta)$ for any $f$ whenever the integrals are defined;
$g\Pi$ is also denoted $g \,\mathrm{d}\Pi$ or $\Pi\circ g^{-1}$ by some authors.

In this paper, we define the posterior on $\theta$, $\Pi(\cdot\mid x)$,
by $\pi(\theta\mid x) \propto f(x\mid\theta) \pi(\theta)$.
Thus, the posterior $\Pi(\cdot\mid x)$ may be proper or improper.
There are three possible cases. First, if we use a proper prior, by
applying the Bayes formula, we obtain a posterior which is a
probability measure.
If the prior is an improper measure such that $\int_{\Theta
}f(x\mid\theta)\pi(\theta)\,\dif\theta< +\infty$, we can formally
apply the Bayes rule, which provides a posterior probability measure by
renormalization.
At last, if the prior is an improper measure such that $\int_{\Theta
}f(x\mid\theta)\pi(\theta)\,\dif\theta= +\infty$, the posterior
is an improper measure defined by $\pi(\theta\mid x) = f(x\mid\theta)
\pi(\theta)$
up to within a scalar factor.

In this section, we study the consequences of the $q$-vague convergence
of $\{\Pi_n\}_n$ on the posterior analysis.
In the general case where the posteriors may be proper or improper, we
give a result about the $q$-vague convergence of posteriors $\{\Pi
_n(\cdot\mid x)\}_n$ to $\Pi(\cdot\mid x)$.
When posteriors are probability measures, we can establish results
about the narrow convergence instead of the $q$-vague convergence.

%
\begin{Prop}\label{convpost}
Let $\{\Pi_n\}_n$ be a sequence of priors which converges $q$-vaguely
to $\Pi$.
Assume that, $\theta\longmapsto f(x\mid\theta)$ is a non-zero
continuous function on $\Theta$.
Then $\{\Pi_n(\cdot\mid x)\}_n$ converges $q$-vaguely to $\Pi(\cdot\mid x)$.

Moreover, if $\{\Pi_n(\cdot\mid x)\}_n$ is a tight sequence of
probabilities and $\Pi(\cdot\mid x)$ is a probability,
then $\{\Pi_n(\cdot\mid x)\}_n$ converges narrowly to $\Pi(\cdot\mid x)$.
\end{Prop}

\begin{pf}
Assume that $\{\Pi_n\}_n$ converges $q$-vaguely to $\Pi$.
From Definition~\ref{caractconv}, there exists a sequence of positive
scalars $\{a_n\}_n$ such that $\{a_n\Pi_n\}_n$ converges vaguely to
$\Pi$.
So, for any $h\in\mathcal{C}_K$, $\lim_{n\to\infty} a_n
\Pi_n(h) = \Pi(h)$.
Since $f(x\mid\cdot)$ is a continuous function, $f(x\mid\cdot)h\in
\mathcal
{C}_K$ and $\lim_{n\to\infty} a_n \Pi_n(f(x\mid\cdot)h) = \Pi(f(x\mid
\cdot)h)$. But $ \Pi_n(f(x\mid\cdot)h) = f(x\mid\cdot
)\Pi_n(h) $ and
$\Pi(f(x\mid\cdot)h) = f(x\mid\cdot)\Pi(h)$.
So, $\{a_n f(x\mid\cdot)\Pi_n\}$ converges vaguely to $f(x\mid\cdot
)\Pi$,
or equivalently $\{f(x\mid\cdot)\Pi_n\}_n$ converges $q$-vaguely to
$f(x\mid\cdot)\Pi$.

If $\{\Pi_n(\cdot\mid x)\}_n$ is a tight sequence of probabilities and
$\Pi(\cdot\mid x)$ is a probability, the second result follows from
Proposition~\ref{lienconvq-vetnarrow}.
\end{pf}

%
\begin{Rem}
If $\Theta$ is discrete, then $f(x\mid\theta)$ is necessary continuous
for the discrete topology.
\end{Rem}

The following results are based on Proposition~\ref{convpost} with
easier-to-check assumptions.

%
\begin{Cor}\label{suitecroissposterior}
Let $\{\Pi_n\}_n$ and $\Pi$ be priors.
Assume that:
\begin{longlist}[(2)]
\item[(1)] there exists a sequence of positive real numbers $\{a_n\}_{n}$ such
that the sequence $\{a_n\pi_n\}_{n}$ converges pointwise to $\pi
$,

\item[(2)] $\{a_n \pi_n(\theta)\}_n$ is non-decreasing for all $\theta\in
\Theta$,

\item[(3)] $\theta\longmapsto f(x\mid\theta)$ is continuous and
positive,

\item[(4)] all the posteriors $\Pi_n(\cdot\mid x)$ and $\Pi(\cdot\mid x)$ are
proper.
\end{longlist}
Then, $\{\Pi_n(\cdot\mid x)\}_n$ converges narrowly to $\Pi(\cdot\mid x)$.
\end{Cor}

\begin{pf}
The sequence $\{a_n f \pi_n\}_n$ is a non-decreasing sequence of
non-negative functions. By~monotone convergence theorem,
$ \lim_{n\to\infty} \int a_n f(x\mid\theta) \pi_n(\theta
) \,\dif\theta= \int\lim_{n\to\infty} a_n f(x\mid\theta) \*\pi
_n(\theta) \,\dif\theta=\int f(x\mid\theta) \pi(\theta) \,\dif\theta$.
So, $\{a_n\Pi_n(f)\}_n$ converges to $\Pi(f)>0$.
So there exists $N$ such that for all $n>N$, $a_n\Pi_n(f) \geq
\frac{1}{2}\Pi(f)$.
Consider $\{K_m\}_m$ an increasing sequence of compact sets such that
$\bigcup K_m=\Theta$.
The sequence $\{K_m^c\}_m$ decreases to $\varnothing$ so $\lim_{m\to
\infty}\Pi(f\mathds{1}_{K_m^c}) = 0$.
Thus, for all $\varepsilon>0$, there exists $M$ such that, for all
$m\geq M$, $\Pi(f\mathds{1}_{K_m^c})\leq\varepsilon$.
So, for all $n>N$, $\frac{f \Pi_n(K_M^c)}{\Pi_n(f)} = \frac{f a_n\Pi
_n(K_M^c)}{a_n\Pi_n(f)} \leq \frac{2 a_n\Pi_n(f\mathds{1}_{K_M^c})}{\Pi
(f)} \leq \frac{2\Pi(f\mathds{1}_{K_M^c})}{\Pi(f)} \leq \frac
{2\varepsilon}{\Pi(f)}$.
The second inequality comes from assumption~(3).
Thus, $\{\frac{f \Pi_n}{\Pi_n(f)}\}_n$ is tight.
The result follows from Proposition~\ref{lienconvq-vetnarrow}.
\end{pf}

%
\begin{Cor}\label{convdompost}
Let $\{\Pi_n\}_n$ and $\Pi$ be priors.
Assume that:

(1) there exists a sequence of positive real numbers $\{a_n\}_{n}$ such
that the sequence $\{a_n\pi_n\}_{n}$ converges pointwise to $\pi
$,

(2) there exists a continuous function $g\dvtx\Theta\rightarrow\mathbb
{R}^+$ such that $fg$ is Lebesgue integrable and for all $n\in\N$ and
$\theta\in\Theta$, $a_n\pi_n(\theta)<g(\theta)$,

(3) $\theta\longmapsto f(x\mid\theta)$ is continuous and
positive,

(4) all the posteriors $\Pi_n(\cdot\mid x)$ and $\Pi(\cdot\mid x)$ are
proper.

Then $\{\Pi_n(\cdot\mid x)\}_n$ converges narrowly to $\Pi(\cdot\mid x)$.
\end{Cor}

\begin{pf}
From Proposition~\ref{corconvdensite}, assumptions (1) and (2) imply
that $\{\Pi_n\}_n$ converges \mbox{$q$-}va\-guely to $\Pi$.
From assumption (2), for all $n$, $a_n f(x\mid\theta) \pi_n(\theta
)\leq f(x\mid\theta)g(\theta)$.
Since $fg$ is Lebesgue integrable, by dominated convergence theorem,
$ \lim_{n\to\infty}\int a_n f(x\mid\theta) \pi_n(\theta
) \,\dif\theta= \int\lim_{n\to\infty} a_n f(x\mid\theta) \pi_n(\theta)
\,\dif\theta=\int f(x\mid\theta) \pi(\theta) \,\dif\theta$.
Thus, $\{a_n\Pi_n(f)\}_n$ converges to $\Pi(f)>0$ so there exists $N$
such that for all $n>N$,
$a_n\Pi_n(f) \geq\frac{1}{2}\Pi(f)$.
Consider $\{K_m\}_{m\in\N}$ an increasing sequence of compact sets
such that $\bigcup K_m = \Theta$.
The sequence $\{K_m^c\}_{m\in\N}$ decreases to $\varnothing$ so
$\lim_{m\to\infty}\lambda(fg\mathds{1}_{K_m^c})=0$.
Thus, for all $\varepsilon>0$, there exists $M$ such that for all
$m\geq M$, $\lambda(fg\mathds{1}_{K_m^c})\leq\varepsilon$.
So,\vspace*{2pt} for all $n>N$,
$\frac{f a_n\Pi_n(K_M^c)}{a_n\Pi_n(f)} \leq \frac{2 a_n\Pi_n(f\mathds
{1}_{K_M^c})}{\Pi(f)} \leq \frac{2\lambda(fg\mathds{1}_{K_M^c})}{\Pi
(f)} \leq \frac{2\varepsilon}{\Pi(f)}$.
Thus, $\{\Pi_n(\cdot\mid x)\}_n$ is a tight sequence of probabilities.
The result follows from Proposition~\ref{convpost}.
\end{pf}

The following result will be useful to explain the Jeffreys--Lindley
paradox (see Section~\ref{section5}).

%
\begin{Cor}\label{convposteriorquandpriorconvvague}
Consider a sequence of probabilities $\{\Pi_n\}_n$ which converges
vaguely to the proper measure $\Pi$.
Assume that:

\begin{longlist}[(2)]
\item[(1)] $\theta\longmapsto f(x\mid\theta)$ is continuous and
non-negative,

\item[(2)] $f(x\mid\cdot)\in\mathcal{C}_0(\Theta)$.
\end{longlist}
Then $\{\Pi_n(\cdot\mid x)\}_n$ converges narrowly to $\Pi(\cdot\mid x)$.
\end{Cor}

\begin{pf}
Since the $\Pi_n$ and $\Pi$ are proper measures and $f(\cdot\mid\theta
)$ is a p.d.f., $\Pi_n(\cdot\mid x)$ and $\Pi(\cdot\mid x)$ are probabilities.
We assume that $\{\Pi_n\}_n$ converges vaguely, and so $q$-vaguely, to
$\Pi$ and that $f$ satisfies~(1).
So, from Proposition~\ref{convpost}, $\{\Pi_n(\cdot\mid x)\}_n$
converges $q$-vaguely to $\Pi(\cdot\mid x)$.
From Lemma~\ref{lienconvvagueconvetroite}, $\{\Pi_n(f)\}_n$
converges to $\Pi(f)$.
So, there exists $N$ such that for $n>N$, $\Pi_n(f)>\frac{\Pi(f)}{2}$.
Moreover, from assumption (2), for all $\varepsilon>0$, there exists a
compact $K$ such that for all $\theta\in K^c$, $f(\theta\mid x)\leq
\varepsilon$.
Thus, for all $n>N$,
$\frac{f\Pi_n(K^c)}{\Pi_n(f)} \leq \frac{2\Pi_n(f\mathds{1}_{K^c})}{\Pi
(f)} \leq \frac{2\varepsilon}{\Pi(f)}$.
Thus, $\{\frac{f \Pi_n}{\Pi_n(f)}\}_n$ is tight.
The result follows from Proposition~\ref{convpost}.
\end{pf}

Now, we establish some links between the $q$-vague convergence of $\{
\Pi_n\}_n$ and the convergence of the Bayes estimates $\mathbb
{E}_{\Pi_n}(\theta\mid x)$.

%
\begin{Prop}\label{convestim}
Let $\{\Pi_n\}_n$ be a sequence of priors which converges $q$-vaguely
to $\Pi$.
Assume that:
\begin{longlist}[(2)]
\item[(1)] $\theta\longmapsto f(x\mid\theta)$ is a non-zero continuous function
on $\Theta$,

\item[(2)] the family $\{\Pi_n(\cdot\mid x)\}_n$ is a family of probabilities
uniformly integrable (see Billingsley \cite{MR0233396}, page~32).
\end{longlist}
Then $\lim_{n\to\infty}\mathbb{E}_{\Pi_n}(\theta\mid x)=
\mathbb{E}_{\Pi}(\theta\mid x)$.
\end{Prop}

\begin{pf}
From Proposition~\ref{convpost}, $\{\Pi_n(\theta\mid x)\}_n$ converges
$q$-vaguely to $\Pi(\theta\mid x)$.
For all $n$, $\Pi_n(\cdot\mid x)$ and $\Pi(\cdot\mid x)$ are probability
measures and $\{\Pi_n(\cdot\mid x)\}_n$ uniformly integrable implies that
$\{\Pi_n(\cdot\mid x)\}_n$ is tight.
So, from Proposition~\ref{convpost}, $\{\Pi_n(\theta\mid x)\}_n$
converges narrowly to $\Pi(\theta\mid x)$.
The result follows from Billingsley (\cite{MR0233396}, Theorem 5.4).
\end{pf}

We give an other version of Proposition~\ref{convestim} with a more
restrictive but easier-to-check condition than uniform integrability.

%
\begin{Cor}\label{convestimvar}
Let $\{\Pi_n\}_n$ be a sequence of priors which converges $q$-vaguely
to $\Pi$.
Assume that $\theta\longmapsto f(x\mid\theta)$ is a non-zero
continuous function on $\Theta$, and that $\{\Pi_n(\cdot\mid x)\}_n$ is
a family of probabilities such that
$\{\Var_{\Pi_n}(\theta\mid x)\}_n$ is bounded above.
Then $\lim_{n\to\infty}\mathbb{E}_{\Pi_n}(\theta\mid x) = \mathbb
{E}_{\Pi}(\theta\mid x)$.
\end{Cor}

\begin{pf}
This is a consequence of Billingsley (\cite{MR0233396}, page~32) and Proposition~\ref{convestim}.
\end{pf}

\section{Some constructions of sequences of vague priors}\label{nvellesection3}

In this section, we give some constructions of sequences of probability
measures that approximate a given improper prior such as the Haar
measures or the Jeffreys prior.
We have shown in the proof of Proposition~\ref{espquotient} that any
improper prior may be approximated by truncation. Here, we give other
constructions for the Haar measure or the Jeffreys prior.

\subsection{Location and scale models}\label{subsubsectionlocation}

The parameter $\theta$ is said to be a location parameter if there
exists a p.d.f. $g$ such that $f(x\mid\theta)=g(x-\theta)$.
For instance, it is the case when $X|\theta\sim\mathcal{N}(\theta
,\sigma^2)$ with known $\sigma^2$.
The underlying group is $(\R,+)$ and the Haar measure $\lambda_{\R}$
is improper.

%
\begin{Prop}\label{HaarsurR,+}
Let $\Pi$ be a continuous probability measure on $\R$.
Assume that the p.d.f. $\pi(\theta)$ of $\Pi$ with respect to the
Lebesgue measure $\lambda_{\R}$ is bounded above by a continuous and
increasing function and is continuous at $\theta= 0$ with $\pi(0)>0$.
We define $\Pi_n$ by $\pi_n(\theta) = \frac{1}{n}\pi(\frac
{\theta}{n})$.
Then, $\{\Pi_n\}_{n>0}$ converges $q$-vaguely to $\lambda_{\R}$.
\end{Prop}

\begin{pf}
Put $\pi_n(\theta)=\frac{1}{n}\pi(\frac{\theta}{n})$. Put
$a_n=n$, then $\lim_{n\to\infty} a_n\pi_n(\theta) = \lim_{n\to\infty}\pi
(\frac{\theta}{n}) = \pi(0)>0$
since $\pi$ is continuous at $0$.
Moreover, $\pi$ is bounded above by a continuous and increasing
function, so there exists $g$ such that,
for all $\theta\in\R$ and for all $n>0$, $\pi(\frac{\theta}{n}) \leq
g(\frac{\theta}{n}) \leq g(\theta)$.
The result follows from Proposition~\ref{corconvdensite}.
\end{pf}

We note that Hartigan \cite{MR1389885} used a dual approach. He reduced the
influence of the prior by letting the conditional variance $\sigma^2$
reducing to $0$.
He arrived at similar conclusions.
He assumed that $\Pi$ is locally uniform at $0$, but it is equivalent
to assuming that $\Pi$ is continuous and positive at $0$.
We replace his condition ``$\pi$ tail-bounded'' by the condition
``$\pi$ bounded''.

%
\begin{Rem}
Proposition~\ref{HaarsurR,+} holds with the assumption ``$\pi$
bounded'' instead of ``$\pi$ bounded above by a continuous and
increasing function''.
\end{Rem}

We now study the scale model.
The strictly positive parameter $\sigma$ is said to be a scale parameter
if $f(x|\sigma)=\frac{1}{\sigma}g(\frac{x}{\sigma})$ where $g$ is
a p.d.f.
If $\sigma$ is a scale parameter, $\log(\sigma)$ is a location
parameter for $\log(X)$.
Here, the concerned group is $(\R^+\setminus\{0\},\times)$ and the
Haar measure $\frac{1}{\sigma}\lambda_{\R^+\setminus\{0\}}$ is improper.
The following proposition is the equivalent of Proposition~\ref
{HaarsurR,+} for the Haar measure on $(\R^+\setminus\{0\},\times)$.

%
\begin{Cor}\label{propscalemodel}
Let $\Pi$ be a continuous probability measure on $\R^+\setminus\{0\}$.
Assume that the p.d.f. $\pi(\sigma)$ of $\Pi$ with respect to the
Lebesgue measure $\lambda_{\R^+\setminus\{0\}}$ is bounded above by
a continuous and increasing function\vspace*{1pt} and is continuous at $\sigma=1$
with $\pi(1)>0$.
We define $\Pi_n$ by $\pi_n(\sigma) = \frac{1}{n}\sigma^{\sfrac
{1}{n}-1}\pi(\sigma^{\sfrac{1}{n}})$.
Then $\{\Pi_n\}_{n>0}$ converges $q$-vaguely to $\frac{1}{\sigma
}\lambda_{\R^+\setminus\{0\}}$.
\end{Cor}

\begin{pf}
Put $\theta= \log(\sigma)$.
From Proposition~\ref{chgmtparam}, $\widetilde{\pi}(\theta) = \ee^{\theta
}\pi(\ee^{\theta})$ which is bounded above by the continuous
and increasing function $\ee^{\theta}g(\ee^{\theta})$.
The result follows from Proposition~\ref{HaarsurR,+}.
\end{pf}

\subsection{Jeffreys conjugate priors (JCPs)}\label{section253}
The Jeffreys prior is one of the most popular prior when no information
is available, but in many cases, is improper. Consider that the
distribution $X|\theta$ belongs to an exponential family,
{i.e.}, $f(x\mid\theta)=\exp\{\theta\cdot t(x)-\phi(\theta)\}
h(x)$, for some functions $t(x)$, $h(x)$ and $\phi(\theta)$, and
$\theta\in\Theta$, where $\Theta$ is an open set in $\mathbb R^p$,
$p\geq1$,
such that $f(x\mid\theta)$ is a well-defined p.d.f.
We assume that $\phi(\theta)$ and $I_\theta(\theta)$ are
continuous. These conditions are satisfied if $t(X)$ is not
concentrated on an hyperplane almost surely (see Barndorff-Nielsen \cite{MR489333}).
Druilhet and Pommeret \cite{MR3000019} proposed a class of conjugate priors that
aims to approximate the Jeffreys prior and that is invariant with
respect to smooth reparameterization.
The notion of approximation was defined only from an intuitive point of
view. We can now give a more rigorous approach by using the $q$-vague
convergence.

Denote by $\pi^J(\theta)=|I_\theta(\theta)|^{1/2}$ the p.d.f. of the
Jeffreys prior with respect to the Lebesgue measure,
where $\theta$ is the natural parameter of the exponential family and
$I_{\theta}(\theta)$ is the determinant of the Fisher information
matrix. The JCPs are defined through their p.d.f. with respect to the
Lebesgue measure by
\[
\pi^J_{\alpha,\beta}(\theta)\propto\exp\bigl\{\alpha.\theta-\beta
\phi(\theta)\bigr\} \bigl\llvert I_{\theta}(\theta)\bigr\rrvert
^{\sfrac{1}{2}},
\]
and for a smooth reparameterization $\theta\rightarrow\eta$ by
\[
\pi^J_{\alpha,\beta}(\eta)\propto\exp\bigl\{\alpha.\theta(\eta)-
\beta\phi\bigl(\theta(\eta)\bigr)\bigr\} \bigl\llvert I_{\eta}(\eta
)\bigr
\rrvert^{\sfrac{1}{2}}.
\]

%
\begin{Prop}
\label{propfamilleexpo}
Let $\{(\alpha_n,\beta_n)\}_n$ be a sequence of real numbers that
converges to $(0,0)$. Then, for the natural parameter $\theta$ or for
any smooth reparameterization $\eta$,
$\{\Pi^J_{\alpha_n,\beta_n}\}_n$ converges \mbox{$q$-}vaguely to $\Pi^J$.
\end{Prop}

\begin{pf}
Choose $\{a_n\}_n$ such that $a_n \pi^J_{\alpha_n,\beta_n}(\theta
)=\exp\{\alpha_n \theta-\beta_n\phi(\theta)\} |I_{\theta
}(\theta)|^{\sfrac{1}{2}}$, which converges pointwise to $|I_{\theta
}(\theta)|^{\sfrac{1}{2}}$.
Put $\gamma_n=(\alpha_n,\beta_n)$ and $\psi(\theta) = (\theta
,-\phi(\theta))$. We have $\gamma_n\cdot\psi(\theta)=\alpha_n
\theta-\beta_n\phi(\theta)$. By Cauchy--Schwarz
inequality, $\gamma_n\cdot\psi(\theta) \leq\llVert \gamma_n\rrVert
\llVert
\psi(\theta)\rrVert $. Since $\gamma_n$ converges to $(0,0)$, there exists
$N$ such that, for $n>N$, $\llVert \gamma_n\rrVert <1$. Let $K$ be a
compact set
in $\Theta$,
by continuity of $\psi(\theta)$, since $\phi(\theta)$ is
continuous, and by continuity of $I_{\theta}(\theta)$, there exist
$M_1$ and $M_2$
such that, for all $\theta\in K$, $\llVert \psi(\theta)\rrVert
<M_1$ and
$|I_{\theta}(\theta)|^{\sfrac{1}{2}}<M_2$. Therefore, $a_n \pi
^J_{\alpha_n,\beta_n}(\theta)\leq M_2 \exp\{M_1\}$.
The result follows from Proposition~\ref{thmconvdensite}.
\end{pf}

Even if we have the convergence to the Jeffreys prior, we have no
guaranty that $\Pi^J_{\alpha_n,\beta_n}$ is a proper prior and there
is no general result to characterize this property
such as in Diaconis and Ylvisaker \cite{MR520238} for usual conjugate priors.
For example, consider inverse Gaussian models with likelihood $ f(x;\mu
,\lambda)= (\frac{\lambda}{2\uppi x^3} )^{\sfrac{1}{2}}\exp
(\frac{-\lambda(x-\mu)^2}{2\mu^2x} )\mathds{1}_{\{x>0\}}$
where $\mu>0$ denotes the mean parameter and $\lambda>0$ stands for
the shape parameter.
Considering the parameterization $ (\psi=\frac{1}{\mu},\lambda
)$, the JCPs are given by
$\pi^J_{\alpha,\beta}(\psi,\lambda) \propto \ee^{-\sklfrac{\lambda}{2}(\alpha_1\psi^2-2\beta\psi+\alpha_2)} \psi^{-\sfrac
{1}{2}}\lambda^{\sfrac{(\beta-1)}{2}}$.
Druilhet and Pommeret \cite{MR3000019} showed that $\pi^J_{\alpha,\beta}(\psi
,\lambda)$ is proper iff $\alpha_1>0$, $\alpha_2>0$ and $-\frac
{1}{2}\leq\beta<\sqrt{\alpha_1\alpha_2}$.
So, we may consider the sequences $\alpha_{1,n}=\alpha_{2,n}=\frac
{1}{n}$ and $\beta_{n}=\frac{1}{2n}$. By\vspace*{1pt} Proposition~\ref{propfamilleexpo},
$\Pi^J_{\alpha_n,\beta_n}(\psi,\lambda)$ is therefore a sequence
of proper priors that converges $q$-vaguely to the Jeffreys prior $ \Pi^J$.

%
\begin{Rem}
For any continuous function $g$ on $\Theta$, we can define $\pi
^g_{\alpha,\beta} (\theta) \propto\exp\{\alpha.\theta- \beta\phi
(\theta)\} g(\theta)$ and
$\pi^g(\theta) = g(\theta)$.
Similarly to Proposition~\ref{propfamilleexpo}, it can be shown that
$\{\Pi^g_{\alpha_n,\beta_n}\}$ converges $q$-vaguely to $\Pi^g$.
\end{Rem}

\section{Some examples}\label{sectionexemples}

In this section, we consider some usual distributions and we look at
the $q$-vague limiting measure.

\subsection{Approximation of flat prior from Uniform distributions}
\subsubsection{The discrete case}

Consider $\Theta=\N$, and $\Pi_n=\mathcal{U}(\{0,1,\ldots,n\})$
the Uniform distribution on $\{0,\ldots,n\}$.
Then $\{\Pi_n\}_n$ converges $q$-vaguely to the counting measure.

Indeed, $\pi_n(\theta) = \frac{1}{n+1} \mathds{1}_{\{
0,1,\ldots,n\}}(\theta)$.
Put $a_n = n+1$, then, for $\theta\in\N$, $\lim_{n\to\infty} a_n\pi
_n(\theta) = \lim_{n\to\infty
} \mathds{1}_{\{0,1,\ldots,n\}}(\theta) = 1$.
The result follows from Proposition~\ref{caractconvmesdiscrete}.

\subsubsection{The continuous case}

Let $\Theta=\mathbb R$, and $\Pi_n=\mathcal{U}([-n,n])$ the Uniform
distribution on $[-n,n]$.
Then $\{\Pi_n\}_n$ converges $q$-vaguely to the Lebesgue measure
$\lambda_{\mathbb{R}}$.

It corresponds to a location model so the result follows from
Proposition~\ref{HaarsurR,+} with $\Pi= \mathcal{U}([-1,1])$.

\subsection{Poisson distribution}

Here is an example where a family of proper priors does not converge
$q$-vaguely.
Let $\Theta=\N$ and $\Pi_n$ be the Poisson distribution with $\pi
_n(\theta)=\exp(-n)\frac{n^{\theta}}{\theta!}$.
Assume that there exists $\Pi$ such that $\{\Pi_n\}_n$ converges
$q$-vaguely to $\Pi$. Then, from Proposition~\ref{caractconvmesdiscrete},
there exists a sequence $\{a_n\}_n$ such that for all $\theta\in
\Theta$, $\lim_{n\to\infty} a_n \pi_n(\theta) = \pi(\theta)$.
Consider\vspace*{2pt} $\theta_0\in\Theta$ such that $\pi(\theta_0)>0$. There
exists $N$ such that, for all $n>N$, $\pi_n(\theta_0)>0$.
Consider $\theta>\theta_0$, for all $n>N$, $\frac{\pi_n(\theta
)}{\pi_n(\theta_0)}=\frac{\theta_0!}{\theta!}n^{\theta-\theta_0}$
and $\lim_{n\to\infty}\frac{\pi_n(\theta)}{\pi_n(\theta
_0)} = \frac{\pi(\theta)}{\pi(\theta_0)} < +\infty$.
On\vspace*{2pt} the other side, $\lim_{n\to\infty}\frac{\theta
_0!}{\theta!}n^{\theta-\theta_0} = +\infty$.
This is a contradiction. So, there is no prior $\Pi$ such that $\{\Pi
_n\}_n$ converges $q$-vaguely to $\Pi$.

\subsection{Normal distribution}\label{exnormaledistrib}

Let $\Theta=\mathbb R$ and $ \Pi_n=\mathcal{N}(0,n)$ the normal
distribution with zero mean and variance equal to $n$.
Then $\{\Pi_n\}_n$ converges $q$-vaguely to the Lebesgue measure on
$\R$.

Indeed, $\pi_n(\theta)=\frac{1}{\sqrt{2\uppi n}}\ee^{-\sfrac{\theta
^2}{2 n}}$ and $\pi(\theta)=1$.
Put $a_n=\sqrt{2\uppi n}$, $n>0$.
Then $\{a_n \pi_n\}_{n>0}$ converges pointwise to $1$.
Moreover, for all $n$ and all $\theta$, $a_n\pi_n(\theta)<2$.
The result follows from Proposition~\ref{corconvdensite}.

%
\begin{Rem}\label{remarqueunicitelimite+divvar}
From Theorem~\ref{unicitelimite}, $\{\mathcal{N}(0,n)\}_{n>0}$
cannot converge to another limiting measure than the Lebesgue measure
(up to within a scalar factor).
\end{Rem}

More generally, it can be shown that the limiting measure is the same
for $\{\mathcal{N}(\mu_n,n)\}_{n}$ where $\{\mu_n\}_n$ is a constant
or a bounded sequence.
So, we consider now the case where $\lim_{n\to\infty}\mu_n = +\infty$
by taking $\mu_n=n$.

%
\begin{Prop}
We have three cases for the convergence of $\mathcal{N}(n,\sigma_n^2)$:
\begin{longlist}[2.]
\item[1.] If $\lim_{n\to\infty}\frac{n}{\sigma_n^2} = +\infty$, then $\{
\mathcal{N}(n,\sigma_n^2)\}_n$ does not converge
$q$-vaguely.

\item[2.] If $\lim_{n\to\infty}\frac{n}{\sigma_n^2} = c$
with $0<c<\infty$, then $\{\mathcal{N}(n,\sigma_n^2)\}_n$ converges
$q$-vaguely to $\ee^{c\theta}\,\dif\theta$.

\item[3.] If $\lim_{n\to\infty}\frac{n}{\sigma_n^2} = 0$,
then $\{\mathcal{N}(n,\sigma_n^2)\}_n$ converges $q$-vaguely to
$\lambda_{\R}$.
\end{longlist}
\end{Prop}

\begin{pf}
For all $n>0$, we denote by $\Pi_n=\mathcal{N}(n,\sigma_n^2)$, and
by $\pi_n$ the p.d.f. with respect to the Lebesgue measure,
$\pi_n(\theta)=\frac{1}{\sqrt{2\uppi}\sigma_n}\exp(-\frac{(\theta
-n)^2}{2\sigma_n^2})$.
\begin{longlist}[2.]
\item[1.] Put $\widetilde{\pi}_n(\theta)=\exp(-\frac{\theta
^2}{2\sigma_n^2}+\frac{\theta n}{\sigma_n^2} )$ and
$\widetilde{\pi}(\theta)=\ee^{\afrac{n^2}{2\sigma_n^2}}\pi(\theta)$.
So $\{\Pi_n\}_n$ converges $q$-vaguely iff $\{\widetilde{\Pi}_n\}_n$
converges $q$-vaguely.
Assume\vspace*{2pt} that there exists $\widetilde{\Pi}$ such that $\{\widetilde
{\Pi}_n\}_n$ converges $q$-vaguely to~$\widetilde{\Pi}$.
Then there exists a sequence $\{a_n\}_n$ such that $\{a_n\widetilde
{\Pi}_n\}_n$ converges vaguely\vspace*{2pt} to $\widetilde{\Pi}$.
Since $\widetilde{\Pi}\neq0$, there exists an interval $[A_1,A_2]$
such that $-\infty<A_1<A_2<+\infty$ and $0<\widetilde{\Pi
}([A_1,A_2])<+\infty$.
Consider $[B_1,B_2]$ such that $A_2<B_1<B_2<+\infty$.
There exists $N$ such that for $n>N$, $\theta\longmapsto-\frac
{\theta^2}{2n}+\frac{\theta n}{\sigma_n^2}$ is non-decreasing.
For a such $n$, $\widetilde{\Pi}_n([B_1,B_2]) \geq(B_2-B_1) \exp(-\frac
{B_1}{2\sigma_n^2}+\frac{B_1 n}{\sigma
_n^2})$ and
$\widetilde{\Pi}_n([A_1,A_2]) \leq (A_2-A_1)\exp(-\frac
{A_2}{2\sigma_n^2}+\frac{A_2 n}{\sigma_n^2})$.
So $\frac{\widetilde{\Pi}_n([B_1,B_2])}{\widetilde{\Pi
}_n([A_1,A_2])} \geq\frac{B_2-B_1}{A_2-A_1}\exp(C(n))$
with $C(n) = \frac{n (B_1-A_2)}{\sigma_n^2}-\frac
{(B^2_1-A_2^2)}{2\sigma_n^2} \geq\frac{n(B_1-A_2)}{2\sigma_n^2}$.
Thus, $\lim_{n\to\infty} \frac{\widetilde{\Pi
}_n([B_1,B_2])}{\widetilde{\Pi}_n([A_1,A_2])} = +\infty$ but
$\lim_{n\to\infty} \frac{\widetilde{\Pi
}_n([B_1,B_2])}{\widetilde{\Pi}_n([A_1,A_2])} = \frac{\widetilde
{\Pi}([B_1,B_2])}{\widetilde{\Pi}([A_1,A_2])} < +\infty$.
So, $\{\Pi_n\}_n$ does not converge $q$-vaguely.

\item[2.] Put\vspace*{1pt} $a_n=\frac{1}{\sqrt{2\uppi}\sigma_n}\exp(-\frac
{n^2}{2\sigma_n^2})$.
Then $\lim_{n\to\infty} a_n\pi_n(\theta) = \lim_{n\to\infty}\exp(-\frac
{\theta^2}{2\sigma_n^2}+\frac
{\theta n}{\sigma_n^2}) = \ee^{c\theta}$.
Since $\lim_{n\to\infty}\frac{n}{\sigma_n^2} = c$,
there exists $N$ such that for all $n>N$, $\frac{n}{\sigma_n^2}\in
[c-\varepsilon,c+\varepsilon]$.
So,\vspace*{1pt} for all $n>N$, $\exp(-\frac{\theta^2}{2\sigma_n^2}+\frac
{\theta n}{\sigma_n^2})\leq\exp((c+\varepsilon)\theta)$ which
is continuous.
The result follows from Proposition~\ref{corconvdensite}.

\item[3.] This is the same reasoning as point~2. with $\lim_{n\to
\infty} a_n\pi_n(\theta) = 1$ and $a_n\pi_n(\theta)\leq
1+\varepsilon$ for all $n>N$ and $N$ large enough.\quad\qed
\end{longlist}\noqed
\end{pf}

%
\begin{Ex}\label{exnormalepost}
Assume that $X|\theta\sim\mathcal{N}(\theta,\sigma^2)$,
$\sigma^2$ known, and put the prior $\Pi_n = \mathcal{N}(0,n)$
on $\theta$.
Then $\Pi_n(\theta\mid x) = \mathcal{N}(\frac{n x}{\sigma
^2+n},\frac{\sigma^2 n}{\sigma^2+n})$.
From\vspace*{2pt} Section~\ref{exnormaledistrib}, the two first hypotheses are
satisfied and $\{\mathcal{N}(0,n)\}_n$ converges $q$-vaguely to the
Lebesgue measure $\lambda_{\R}$ so here,
$\Pi= \lambda_{\R}$.
Moreover, $\theta\longmapsto f(x\mid\theta)$ is continuous and
positive on $\Theta$ and $\Pi(\cdot\mid x) = \mathcal{N}(x,\sigma
^2)$ is proper.
So, from Theorem~\ref{suitecroissposterior}, $\{\mathcal{N}(\frac{n
x}{\sigma^2+n},\frac{\sigma^2 n}{\sigma^2+n})\}_n$ converges
narrowly to $\mathcal{N}(x,\sigma^2)$.
\end{Ex}

%
\begin{Ex}
To continue\vspace*{2pt} Example~\ref{exnormalepost}, $\Var_{\Pi_n}(\theta\mid x) =
\frac{\sigma^2 n}{\sigma^2+n}$ is bounded above by $\sigma^2$
and the other hypothesis of Proposition~\ref{convestimvar} have
already been verified in Example~\ref{exnormalepost}. So, from
Proposition~\ref{convestimvar},
$\lim_{n\to\infty}\E_{\Pi_n}(\theta\mid x) = \E_{\Pi
}(\theta)$.
Indeed, $\lim_{n\to\infty}\E_{\Pi_n}(\theta) = \lim_{n\to\infty}\frac{n
x}{\sigma^2+n} = x = \E_{\Pi
}(\theta)$.
\end{Ex}

\subsection{Gamma distribution}

\subsubsection{Approximation of \texorpdfstring{$\Pi=\frac{1}{\theta}\mathds{1}_{\theta>0}\,\dif\theta$}{$Pi=\frac{1}{theta}\mathds{1}_{theta>0}dtheta$}}\label{sectionexemplesgammavar}
Let $\Theta= \mathbb R_+$ and $\Pi_n = \gamma(\alpha
_n,\beta_n)$ the Gamma distributions with $\lim_{n\to\infty
}(\alpha_n,\beta_n) = (0,0)$.
We have $\pi_n(\theta) = \frac{{\beta_n}^{\alpha_n}}{\Gamma
(\alpha_n)} \theta^{\alpha_n-1} \ee^{-\beta_n \theta}$.
Put $a_n=\frac{\Gamma(\alpha_n)}{{\beta_n}^{\alpha_n}}$. Then
$a_n\pi_n(\theta) = \theta^{\alpha_n-1} \ee^{-\beta_n \theta
}$ and $\{a_n\pi_n(\theta)\}_n$ converges to $\frac{1}{\theta}$.
Put $g(\theta) = \frac{1}{\theta} \mathds{1}_{]0,1]}(\theta
) + \mathds{1}_{]1,+\infty[}(\theta)$.
The sequence $\{\alpha_n\}_n$ goes to $0$ so there exists $N$ such
that for all $n>N$, $\alpha_n<1$.
So, for $n>N$ and for $\theta>0$, $a_n \pi_n(\theta) \leq \theta
^{\alpha_n-1} \leq g(\theta)$.\vspace*{1pt}
Since $g$ is a continuous function on $\R^*_+$, from Proposition~\ref
{corconvdensite}, $\{\Pi_n\}_n$ converges $q$-vaguely to~$\frac
{1}{\theta} \,\dif\theta$.\vspace*{1pt}

Recall that for $\theta\sim\gamma(a,b)$, $\E(\theta)=\frac{a}{b}$
and $\Var(\theta)=\frac{a}{b^2}$.
We can see below that the same convergence may be obtained with
different convergences of the mean and variance.
\begin{itemize}
\item For $\Pi_n=\gamma(\frac{1}{n},\frac{1}{n})$, $\E_{\Pi
_n}(\theta)=1$ for all $n$ and $\lim_{n\to\infty}\Var_{\Pi
_n}(\theta)=\lim_{n\to\infty} n=+\infty$.
\item For $\Pi_n=\gamma(\frac{1}{n},\frac{1}{\sqrt{n}})$, $\lim_{n\to
\infty} \E_{\Pi_n}(\theta) = \lim_{n\to\infty} \frac{1}{\sqrt{n}} = 0$
and $\lim_{n\to\infty} \Var_{\Pi_n}(\theta) = 1$ for all~$n$.
\item For $\Pi_n=\gamma(\frac{1}{n},\frac{1}{n^{\sfrac{1}{3}}})$,
$\lim_{n\to\infty} \E_{\Pi_n}(\theta) = \lim_{n\to\infty} n^{-\sfrac
{2}{3}} = 0$ and $\lim_{n\to
\infty} \Var_{\Pi_n}(\theta)=\lim_{n\to\infty} n^{-\sfrac
{1}{3}} = 0$.
\item For $\Pi_n=\gamma(\frac{1}{n},\frac{1}{n^2})$, $\lim_{n\to\infty
}\E_{\Pi_n}(\theta) = \lim_{n\to\infty}
n = +\infty$ and $\lim_{n\to\infty}\Var_{\Pi
_n}(\theta) = \lim_{n\to\infty} n^3 = +\infty$.

\item For $\Pi_n=\gamma(\frac{1}{n},\frac{1}{n^{\sfrac{2}{3}}})$,
$\lim_{n\to\infty}\E_{\Pi_n}(\theta) = n^{-\sfrac
{1}{2}} = 0$ and $\lim_{n\to\infty}\Var_{\Pi_n}(\theta
) =\break \lim_{n\to\infty} n^{\sfrac{1}{3}} = +\infty$.
\end{itemize}

More generally, if $\liminf_n\E_{\Pi_n}(\theta) > 0$ then $\lim_{n\to
\infty}\Var_{\Pi_n}(\theta) = +\infty$,
since $\Var_{\Pi_n}(\theta) = \frac{\E_{\Pi_n}(\theta
)}{\beta_n}$ with $\lim_{n\to\infty}\beta_n = 0$.

\subsubsection{Approximation of 
\texorpdfstring{$\Pi=\frac{1}{\theta}\ee^{-\theta}\mathds{1}_{\theta>0}\,\dif\theta$}{$Pi=\frac{1}{theta}e^{-theta}\mathds{1}_{theta>0}dtheta$}}

Let us show that $\{\gamma(\alpha_n,1)\}$ converges $q$-vaguely to
$\frac{1}{\theta}\ee^{-\theta} \mathds{1}_{\theta>0} \,\dif\theta
$ when $\{\alpha_n\}$ goes to $0$.
Put $\Pi_n= \{\gamma(\alpha_n,1)\}$.
Then $\pi_n(\theta) = \frac{1}{\Gamma(\alpha_n)} \theta
^{\alpha_n-1}\ee^{-\theta}\mathds{1}_{\theta>0}$ is the p.d.f. of $\Pi_n$.
Put $a_n= \Gamma(\alpha_n)$, then $a_n\pi_n(\theta) = \theta
^{\alpha_n-1} \ee^{-\theta} \mathds{1}_{\theta>0}$ converges to
$\pi(\theta) = \frac{1}{\theta} \ee^{-\theta} \mathds
{1}_{\theta>0}$.
Moreover, since $\{\alpha_n\}_n$ goes to $0$, there exists $N$ such
that for $n>N$, $\alpha_n<1$.
Put $g(\theta) = \frac{1}{\theta} \mathds{1}_{]0,1]}(\theta
) + \mathds{1}_{]1,+\infty[}(\theta)$.
So, for $n>N$ and $\theta>0$, $a_n\pi_n(\theta) \leq \theta^{\alpha
_n-1} \leq g(\theta)$.
The function $g$ is continuous so from Proposition~\ref
{corconvdensite}, $\{\gamma(\alpha_n,1)\}_n$ converges $q$-vaguely to
$\frac
{1}{\theta} \ee^{-\theta} \mathds{1}_{\theta>0} \,\dif\theta$.
Since $\lim_{n\to\infty}\alpha_n = 0$, we necessarily
have $\lim_{n\to\infty}\E_{\Pi_n}(\theta) = 0$ and
$\lim_{n\to\infty}\Var_{\Pi_n}(\theta) = 0$.

\section{Convergence of Beta distributions}\label{section4}

We now consider a more complex example which often appears in
literature; see, for example, Tuyl \textit{et~al.} \cite{MR2486242}.
Let $X$ represents the number of successes in $N$ Bernoulli trials, and
$\theta$ be the probability of a success in a single trial.
Since the Beta distribution and the Binomial distribution form a
conjugate pair, a common prior distribution on $\theta$ is $\beta
(\alpha,\alpha)$ which have mean and median equal to $\frac{1}{2}$.
Three ``plausible'' non-informative priors were listed by Berger (\cite{MR804611},
page~89): the Bayes--Laplace prior $\beta(1,1)$, the Jeffreys prior
$\beta(\frac{1}{2},\frac{1}{2})$ and the improper Haldane prior,
wrote down $\beta(0,0)$, whose density is $\pi_H(\theta)=\frac
{1}{\theta(1-\theta)}$ with respect to the Lebesgue measure on $]0,1[$.
If we want $\beta(\alpha,\alpha)$ with large variance, necessarily
$\alpha$ must be close to $0$.
Thus, we choose $\beta(\frac{1}{n},\frac{1}{n})$.
The\vspace*{1pt} density of $\Pi_n=\beta(\frac{1}{n},\frac{1}{n})$ with respect
to the Lebesgue measure on $]0;1[$ is $\pi_n(\theta)=\frac
{1}{B(\sfrac{1}{n},\sfrac{1}{n})} \theta^{\sfrac{1}{n}-1}(1-\theta
)^{\sfrac{1}{n}-1}$.
As mentioned, for example, by Bernardo \cite{MR547240} or Lane and Sudderth
\cite{MR684869}, there are two possible limiting distributions for $\beta(\frac
{1}{n},\frac{1}{n})$ when $n$ goes to $+\infty$.
The first one is $\frac{1}{2}(\delta_0+\delta_1)$ which is the
limiting measure given by the standard probability theory.
The second one is the Haldane prior $\Pi_H$ which is deduced from the
posterior distributions and estimators (Lehmann and Casella \cite{MR1639875}).
We show that it depends on the space where $\theta$ lives.
Choosing $]0,1[$ or $[0,1]$ does not matter for $\beta(\frac
{1}{n},\frac{1}{n})$ but it matters for the limiting distributions.
We may note that the Haldane prior is a Radon measure on $]0,1[$ but
not on $[0,1]$ and that $\frac{1}{2}(\delta_0+\delta_1)$ is not
defined on $]0,1[$.

\subsection{Convergence on $]0,1[$}

In this section, we study the convergences on $]0,1[$ of $\{\beta
(\frac{1}{n},\frac{1}{n})\}_{n>0}$, of the sequence of posteriors and
of the sequence of estimators.

Put $a_n = B(\frac{1}{n},\frac{1}{n})$, then $a_n \pi_n(\theta
) = \theta^{\sfrac{1}{n}-1} (1-\theta)^{\sfrac{1}{n}-1}$ converges
to $\pi_H(\theta) = [\theta(1-\theta)]^{-1}$ and for any
$\theta$ and $n$, $a_n \pi_n(\theta) < 5$.
Therefore, from Theorem~\ref{corconvdensite}, $\{\beta(\frac
{1}{n},\frac{1}{n})\}_{n>0}$ converges $q$-vaguely to $\Pi_H$.

Consider the sequence of posteriors.
The sequence of priors $\{\Pi_n\}_n$ converges $q$-vaguely to $\Pi_H$
and $\theta\longmapsto f(x\mid\theta)$ is continuous on $\Theta$.
Then, from Lemma~\ref{convpost}:
\begin{itemize}
\item if\vspace*{1pt} $x=0$, $\{\Pi_n(\theta\mid x)\}_n$ converges $q$-vaguely to the
improper measures with p.d.f. $\pi(\theta) = (1-\theta)^{N-1} \theta^{-1}$,
\item if $x=N$, $\{\Pi_n(\theta\mid x)\}_n$ converges $q$-vaguely to the
improper measures with p.d.f. $\pi(\theta) = \theta^{N-1}(1-\theta)^{-1}$,
\item if $0<x<N$, $\{\Pi_n(\theta\mid x)\}_n$ converges $q$-vaguely to
$\Pi_H(\theta\mid x) = \beta(x, N-x)$.
\end{itemize}
For $0<x<N$, $\beta(x, N-x)$ is proper and $\theta\longmapsto f(x\mid
\theta)$ is continuous and positive. So, from Theorem~\ref
{suitecroissposterior},
$\{\Pi_n(\theta\mid x)\}_{n>0}$ converges narrowly to $\Pi_H(\theta
|x)=\beta(x, N-x)$.

Consider now the Bayes estimators $\E_{\Pi_n}(\theta\mid x)=\frac{1 + n
x}{2 + n N}$ which tend to $\frac{x}{N}$.
So:
\begin{itemize}
\item If $x=0$, $\lim_{n\to\infty}\E_{\Pi_n}(\theta
|x=0)=0$ whereas $\E_{\Pi_H}(\theta\mid x=0)=\frac{1}{N}$.
\item If $x=N$, $\lim_{n\to\infty}\E_{\Pi_n}(\theta
|x=N)=1$ whereas $\E_{\Pi_H}(\theta\mid x=N)=+\infty$.
\item If $0<x<N$, $\lim_{n\to\infty}\E_{\Pi_n}(\theta
|x)=\frac{x}{N}=\E_{\Pi_H}(\theta\mid x)$.
\end{itemize}
For $x=0$ and $x=N$, $\Pi_H(\cdot\mid x)$ is an improper measure. In this
case, $\E_{\Pi_H}(\theta\mid x) = \int_{\Theta}\theta\,\dif\Pi
_H(\theta\mid x)$.

\subsection{Convergence on $[0,1]$}

In this section, we study the convergences on $[0,1]$ of $\{\beta
(\frac{1}{n},\frac{1}{n})\}_{n>0}$, of the sequence of posteriors and
of the sequence of estimators.

For all $n$ and for $0<t<1$, $\Pi_n([0,t[) + \Pi_n([t,1-t]) + \Pi
_n(]1-t$, $1]) = 1$.
But on $]0,1[$, $\{\beta(\frac{1}{n},\frac{1}{n})\}_{n>0}$ converges
$q$-vaguely to the improper measure $\Pi_H$, so $\lim_{n\to
\infty}\Pi_n([t,1-t]) = 0$.
Moreover, for all $n$, $\Pi_n([0,t[)=\Pi_n(]1-t$, $1])$. Thus, for
all $0<t<1$, $\lim_{n\to\infty}\Pi_n([0,t[)=\frac{1}{2}$.
From Billingsley (\cite{MR830424}, page~192), $\{\beta(\frac{1}{n},\frac{1}{n})\}
_{n>0}$ converges narrowly to $\frac{1}{2}(\delta_0+\delta_1)=\Pi
_{\{0,1\}}$.
By Theorem~\ref{unicitelimite}, $\{\beta(\frac{1}{n},\frac{1}{n})\}
_{n>0}$ cannot converge to an other limit such as the Haldane measure,
which is not a Radon measure on $[0,1]$.

The limit of the posterior distributions can be deduced from the limit
of the prior distributions only for $x=0$ and $x=N$.
\begin{itemize}
\item If $x=0$, $\{\Pi_n(\theta\mid x=0)\}$ converges narrowly to $\Pi
_{\{0,1\}}(\theta\mid x=0)=\delta_0$.
\item If $x=N$, $\{\Pi_n(\theta\mid x=N)\}$ converges narrowly to $\Pi
_{\{0,1\}}(\theta\mid x=N) = \delta_1$.
\item If $0<x<N$, $\{\Pi_n(\theta\mid x)\}$ converges narrowly to $\beta
(x,N-x)$ whereas $\Pi_{\{0,1\}}(\theta\mid x)$ does not exist.
\end{itemize}

Similarly, the limit of the estimators can be deduced from the limit of
the prior distributions only for $x=0$ and $x=N$.
\begin{itemize}
\item If $x=0$, $\lim_{n\to\infty}\E_{\Pi_n}(\theta
|x=0)=0=\E_{\Pi_{\{0,1\}}}(\theta\mid x=0)$.
\item If $x=N$, $\lim_{n\to\infty}\E_{\Pi_n}(\theta
|x=N)=1=\E_{\Pi_{\{0,1\}}}(\theta\mid x=N)$.
\item If $0<x<N$, $\lim_{n\to\infty}\E_{\Pi_n}(\theta
|x)=\frac{x}{N}$ whereas $\E_{\Pi_{\{0,1\}}}(\theta\mid x)$ does not exist.
\end{itemize}

\section{The Jeffreys--Lindley paradox}\label{section5}

Consider the standard Gaussian model $X|\theta\sim\mathcal{N}(\theta
,1)$ and the point null hypothesis $H_0\dvt\theta=0$ tested against
$H_1\dvt\theta\neq0$.
If we use the prior $\pi(\theta) = \frac{1}{2}\mathds
{1}_{\theta=0} + \frac{1}{2}\mathds{1}_{\theta\neq0}$ with
respect to the measure $\delta_0+\lambda_{\R}$,
it corresponds to the mass $\frac{1}{2}$ on $H_0$ and the Laplace
prior on $H_1$.
The posterior probability of $H_0$ is $\Pi(\theta=0|x) = [1+\sqrt
{2\uppi}\ee^{x^2/2}]^{-1}$ so $\Pi(\theta=0|x)\leq [1+\sqrt
{2\uppi} ]^{-1}\approx0.285$ whatever the data are.
An alternative is to use a sequence of proper priors $\{\Pi_n\}_n$
whose p.d.f. are $\pi_n(\theta) = \frac{1}{2} \mathds{1}_{\theta
=0} + \frac{1}{2} \mathds{1}_{\theta\neq0} \frac{1}{\sqrt{2\uppi
}n}\ee^{-\afrac{\theta^2}{2n^2}}$.
With these priors, we have $\pi_n(\theta=0|x)= [1+\sqrt{\frac
{1}{1+n^2}}\ee^{\vafrac{n^2 x^2}{2(1+n^2)}} ]^{-1}$ which converges
to $1$.
This limit differs from the ``non-informative'' answer $ [1+\sqrt
{2\uppi}\ee^{x^2/2} ]^{-1}$ and is considered as a paradox.
In the light of the concept of $q$-vague convergence, this result is
not paradoxal since, as shown in Proposition~\ref{JeffreysParadox},
the sequence of priors $\{\frac{1}{2}\delta_0+\frac{1}{2}\mathcal
{N}(0,n^2)\}_n$ converges vaguely to $\frac{1}{2}\delta_0$,
and, the limiting posterior distribution corresponds to the posterior
of the limit of the prior distributions.
The following proposition generalizes this example.

%
\begin{Prop}\label{JeffreysParadox}
Consider\vspace*{1pt} a partition: $\Theta= \Theta_0 \cup\Theta_1$
where $\Theta_0 = \{\theta_0\}$.
Let $\{\widetilde{\Pi}_n\}_n$ be a sequence of probabilities on
$\Theta$ which converges $q$-vaguely\vspace*{1pt} to the improper measure
$\widetilde{\Pi}$ and such that
$\widetilde{\Pi}_n(\theta_0) = \widetilde{\Pi}(\theta_0) = 0$.
Put $\Pi_n = \rho\delta_{\theta_0} + (1-\rho) \widetilde
{\Pi}_n$ where $0<\rho<1$,
then $\{\Pi_n\}_n$ converges vaguely to $\rho\delta_{\theta_0}$.

Moreover, assume that $\theta\longmapsto f(x\mid\theta)$ is
continuous and belongs to $\mathcal{C}_0$.
Then $\{\Pi_n(\cdot\mid x)\}$ converges narrowly to $\Pi(\cdot\mid x)$.
\end{Prop}

\begin{pf}
From\vspace*{1pt} Definition~\ref{caractconv}, there exists $\{a_n\}_{n}$ such
that $\{a_n\widetilde{\Pi}_n\}_{n}$ converges vaguely\break to $\widetilde
{\Pi}$.
For $g \in\mathcal{C}_K$, $\Pi_n(g) = \rho g(\theta_0) + (1-\rho)
\widetilde{\Pi}_n(g) = \rho g(\theta_0) + \frac{1-\rho}{a_n} a_n
\widetilde{\Pi}_n(g)$.
But\break $\lim_{n\to\infty} a_n\widetilde{\Pi}_n(g) = \widetilde{\Pi}(g) <
\infty$.
So, $\lim_{n\to\infty}\frac{1-\rho}{a_n} a_n \widetilde{\Pi}_n(g) = 0$
since,\vspace*{2pt} from Lemma~\ref{aninf},
$\lim_{n\to\infty} a_n = +\infty$.
Thus, $\lim_{n\to\infty}\Pi_n(g) = \rho g(\theta
_0)$. The first result follows.

The second part is a direct consequence of Theorem~\ref
{convposteriorquandpriorconvvague}.
\end{pf}

In the Proposition~\ref{JeffreysParadox}, it is assumed that $\theta
\longmapsto f(x\mid\theta)\in\mathcal{C}_0(\Theta)$.
Now, we consider the case where the limit of the likelihood $f(x\mid
\theta
)$ when $\theta$ is outside of any compact is not $0$ but $f(x\mid
\theta_0)$.
In that case, the limit of the posterior probabilities is the same as
the limit of the prior probabilities, as stated in the following proposition.

%
\begin{Prop}
Consider the same notations and assumptions of Proposition~\ref
{JeffreysParadox}.
Moreover, assume that $\theta\longmapsto f(x\mid\theta)$ is
continuous and such that for all $\varepsilon>0$, there exists a
compact $K$ such that for all
$\theta\in K^c$, $|f(x\mid\theta)-f(x\mid\theta_0)| \leq \varepsilon$.
Then $\lim_{n\to\infty}\Pi_n(\theta=\theta_0|x) = \Pi
(\theta=\theta_0)$ and $\lim_{n\to\infty}\Pi_n(\theta
\neq\theta_0|x) = \Pi(\theta\neq\theta_0)$.
\end{Prop}

\begin{pf}
By Bayes formula: $\Pi_n(\theta=\theta_0|x)=\frac{\rho f(x\mid\theta
_0)}{\rho f(x\mid\theta_0)+(1-\rho)\int_{\Theta}f(x\mid\theta)\,\dif
\widetilde{\Pi}_n(\theta)}$.
But, for all $\varepsilon>0$, there exists a compact $K$ such that,
for all $\theta\in K^c$, $|f(x\mid\theta)-f(x\mid\theta_0)|\leq
\varepsilon$.
So $\int_{\Theta}f(x\mid\theta)\,\dif\widetilde{\Pi}_n(\theta)=\int_K
f(x\mid\theta)\,\dif\widetilde{\Pi}_n(\theta)+\int_{K^c}f(x\mid\theta
) \,\dif\widetilde{\Pi}_n(\theta)$, where:
\begin{itemize}
\item$(f(x\mid\theta_0)-\varepsilon) \widetilde{\Pi}_n(K^c) \leq \int
_{K^c}f(x\mid\theta)\,\dif\widetilde{\Pi}_n(\theta) \leq (f(x\mid\theta
_0)+\varepsilon) \widetilde{\Pi}_n(K^c)$.
From Proposition~\ref{lemcompact}, $\lim_{n\to\infty
}\widetilde{\Pi}_n(K^c) = 1$. So, $\lim_{n\to\infty
}\int_{K^c}f(x\mid\theta) \,\dif\widetilde{\Pi}_n(\theta) = f(x\mid
\theta_0)$.
\item There exists $g \in\mathcal{C}_K(\Theta)$ such that $0 \leq g
\leq 1$ and $g \mathds{1}_K = 1$.
For a such $g$,
\[
\lim_{n\to\infty}\int_K f(x\mid\theta) \,\dif\widetilde
{\Pi}_n(\theta) \leq \lim_{n\to\infty}\frac{1}{a_n}a_n \int_{\Theta}
g(\theta) f(x\mid\theta) \,\dif\widetilde{\Pi}_n(\theta) = 0
\]
since $\lim_{n\to\infty} a_n\int_{\Theta} g(\theta) f(x\mid\theta)
\,\dif\widetilde{\Pi}_n(\theta) = \int_{\Theta} g(\theta) f(x\mid\theta
) \,\dif\widetilde{\Pi
}(\theta) < +\infty$
and $\lim_{n\to\infty} a_n = +\infty$ from Lemma~\ref{aninf}.
\end{itemize}
Thus, $\lim_{n\to\infty}\Pi_n(\theta=\theta_0|x)=\frac
{\rho f(x\mid\theta_0)}{\rho f(x\mid\theta_0)+(1-\rho)f(x\mid\theta
_0)}=\rho
=\Pi(\theta=\theta_0)$.
\end{pf}

To illustrate this result in a more general case, we consider an
example proposed by Dauxois \textit{et~al.} \cite{MR2273734}.
They consider a model choice between $\mathcal{P}(m)$ the Poisson
distribution, $\mathcal{B}(N,m)$ the Binomial distribution and
$\mathcal{NB}(N,m)$ the Negative Binomial distribution.
These models belong to the general framework of Natural Exponential
Families (NEFs) and are determined by their variance function
$V(m)=am^2+m$ where $m$ is the mean parameter.
Thus, a null value for $a$ relates to the Poisson NEF, a negative one
to the Binomial NEF and a positive one to the Negative Binomial NEF.
The prior distribution chosen on the parameter $a$ is $\Pi_K$ defined by
\[
\Pi_K(a)= \cases{ \displaystyle\frac{1}{3}, &\quad if $a=0$,
\vspace*{3pt}\cr
\displaystyle\frac{1}{3K}, &\quad if $\displaystyle\frac{1}{a}\in
\{ 1,\ldots,K\}$,
\vspace*{3pt}\cr
\displaystyle\frac{1}{3K}, &\quad if $\displaystyle-
\frac{1}{a}\in\{n_0,\ldots,n_0+K-1\}$,}
\]
where $K$ is an hyperparameter.
Note that $\Pi_K(a=0) = \Pi_K(a>0) = \Pi_K(a<0) = \frac{1}{3}$.

Dauxois \textit{et~al.} \cite{MR2273734} showed that the sequence of posterior
distributions does not converge to~$\delta_0$ as in the previous case
but $\Pi_K(a=0|X=x)$, $\Pi_K(a>0|X=x)$ and
$\Pi_K(a<0|X=x)$ converge to the prior probabilities $\Pi_K(a=0)$,
$\Pi_K(a>0)$ and $\Pi_K(a<0)$ whatever the data are when
$K\rightarrow+\infty$.

\begin{appendix}

\section*{Appendix: Properties of the quotient space}

\setcounter{Th}{0}

%
\begin{Prop}\label{Hausdorffspace}
$\overline{\mathcal{R}}$ is a Hausdorff space.
\end{Prop}

\begin{pf}
This proof is based on two results of Bourbaki \cite{MR0205210}.
\begin{longlist}
\item[\textit{Step} 1.] $\mathcal{R}$ is a topological space and
$\Gamma=\{\sigma_\alpha\dvtx \Pi\longmapsto\alpha\Pi$, $\alpha\in
\mathbb{R}_+^*\}$ is a homeomorphism group of $\mathcal{R}$.
We consider the equivalence relation:
$\Pi\sim\Pi'$ iff there exists $\alpha>0$ such that $\Pi=\alpha
\Pi'$, {that is}, there exists $\sigma_\alpha\in\Gamma$
such that $\Pi=\sigma_\alpha(\Pi')$.
So, from Bourbaki (\cite{MR0205210}, Section~I.31), $\sim$ is open.

\item[\textit{Step} 2.] Let us show that $G=\{(\Pi,\alpha\Pi)$, $(\Pi,\alpha
\Pi)\in\mathcal{R}\times\mathcal{R}\}$ which is the graph of $\sim
$ is closed.
Let $\{(\Pi_n,\alpha_n\Pi_n)\}_{n>0}$ be a sequence in $G$ such that
$\lim_{n\to\infty}(\Pi_n,\alpha_n\Pi_n) = (\Pi_0,\Pi'_0)$.
The aim is to show that $(\Pi_0,\Pi'_0)\in G$, {that is},
$(\Pi_0,\Pi'_0)$ takes the form $(\Pi_0,\alpha_0\Pi_0)$ where $
\alpha_0\Pi_0\neq0$.
Since $\Pi_0\neq0$, there exists $f_0\in\mathcal{C}_K$ such that
$\Pi_0(f_0)>0$.
Moreover, $\lim_{n\to\infty}\Pi_n(f_0) = \Pi_0(f_0)$
so there exists $N$ such that for all $n\geq N$, $\Pi_n(f_0)>0$.
For all $n\geq N$, $\lim_{n\to\infty}\alpha_n=\lim_{n\to\infty}\frac
{\alpha_n\Pi_n(f_0)}{\Pi_n(f_0)} = \frac{\Pi'_0(f_0)}{\Pi_0(f_0)} =
\alpha_0$.
Thus, for all $f\in\mathcal{C}_ K$, $\lim_{n\to\infty
}\alpha_n\Pi_n(f) = \alpha_0\Pi_0(f)$ and $\lim_{n\to
\infty}\alpha_n\Pi_n(f) = \Pi'_0(f)$.
Since $\mathcal{R}$ is a Hausdorff space, $\alpha_0\Pi_0(f) = \Pi'_0(f)$.
So, the graph of $\sim$, $G$, is closed.
The result follows from Bourbaki (\cite{MR0205210}, Section I.55).
\end{longlist}
\end{pf}
\end{appendix}

\section*{Acknowledgment}
The authors are grateful to Professors C.P. Robert, J. Rousseau and S.
Dachian for helpful discussions.
We acknowledge the comments from reviewers which resulted in an
improved paper.


%

\printhistory
\end{document}